\documentclass[a4paper,11pt]{article}
\usepackage[utf8x]{inputenc}
\usepackage{a4wide}
\usepackage{amssymb}
\usepackage{amsmath}
\usepackage{amsthm}
\usepackage{latexsym}
\usepackage[mathscr]{euscript}

\theoremstyle{definition}
\newtheorem{definition}{Definition}[section]

\theoremstyle{plain}
\newtheorem{Theorem}[definition]{Theorem}

\newtheorem{Lemma}[definition]{Lemma}

\theoremstyle{remark}
\newtheorem{remark}[definition]{Remark}

\newcommand{\R}{\mathbb R}

\newcommand{\Vol}{\mathrm{Vol}}
\newcommand{\Inj}{\mathrm{Inj}}

\newcommand{\eps}{\varepsilon}

\title{The exponential map of a $C^{1,1}$-metric}

\author{Michael Kunzinger\footnote{University of Vienna, Faculty of Mathematics, 
Oskar-Morgenstern-Platz 1, A-1090 Vienna, Austria
michael.kunzinger@univie.ac.at, roland.steinbauer@univie.ac.at, milena.stojkovic@live.com}, 
Roland Steinbauer\footnotemark[\value{footnote}], Milena Stojkovi\'c\footnotemark[\value{footnote}]}

\begin{document}

\date{}

\maketitle

\begin{abstract}
Given a pseudo-Riemannian metric of regularity $C^{1,1}$ on a smooth manifold, we
prove that the corresponding exponential map is a bi-Lipschitz homeomorphism locally around
any point. We also establish the existence of totally normal neighborhoods in an
appropriate sense. The proofs are based on regularization, combined with methods from comparison
geometry. 

\vskip 1em



\end{abstract}

\section{Introduction}

In (smooth) pseudo-Riemannian geometry, the fact that the exponential map
is a diffeomorphism locally around $0$ is of central importance for many fundamental
constructions. As examples we mention the existence of normal neighborhoods, Riemannian
normal coordinates, geodesic convexity, injectivity and convexity radius, comparison
methods, or also, in the Lorentzian case, causality theory. 

The standard way of proving this result rests on an application of the inverse function
theorem. It is therefore applicable to $C^2$ pseudo-Riemannian metrics, where the
exponential map itself is still $C^1$. On the other hand, the lowest regularity for which
the geodesic equation still has unique solutions in general is $C^{1,1}$ (continuously
differentiable with Lipschitz derivatives). In the literature it is generally held that
$C^{1,1}$ delimits the regularity where one can still reasonably expect the `standard'
results to remain valid. However, for $C^{1,1}$-metrics the exponential map is only
Lipschitz, so the inverse function theorem is no longer applicable. In \cite{C11}, which 
provides a careful analysis of causality theory in minimal regularity, P.\ Chru\'sciel
indicates that certain inverse function theorems for Lipschitz functions might shed light
on this problem. To our knowledge, however, so far there is no proof available in the
literature that the exponential map of a $C^{1,1}$-metric retains the maximal possible
regularity, namely that it is a bi-Lipschitz homeomorphism around $0$. In this work
we supply a proof of this result.

In analogy to the smooth case one may call the image of a (star-shaped) open
neighborhood of $0$ in $T_pM$ where $\exp_p$ is a bi-Lipschitz homeomorphism
a normal neighborhood of $p$. We show that in fact there always exist totally normal 
neighborhoods around any point of the manifold, i.e., open sets that are normal
neighborhoods of each of their elements. 

Our method of proof consists in regularizing the metric locally via convolution with
a mollifier to obtain a net $g_\eps$ of smooth metrics of the same signature. We then use 
methods from comparison geometry to obtain sufficiently strong estimates on the
exponential maps of the regularized metrics to be able to carry the bi-Lipschitz
property through the limit $\eps\to 0$. More precisely, we rely on new comparison methods,   
developed only recently by B.\ L.\ Chen and P.\ LeFloch in their studies
on the injectivity radius of Lorentzian metrics (\cite{Chen}).
In the Riemannian case, one may alternatively use the Rauch comparison theorem,
as well as injectivity radius estimates due to Cheeger, Gromov and Taylor, as 
will be pointed out in Section \ref{Riemsec}.  To show
existence of totally normal neighborhoods we use the uniform estimates derived above
to adapt the standard proof of the smooth case. 

Our notation is standard, cf., e.g., \cite{Jost,ON83}. By $B_h(p,r)$ we denote the open ball
around the point $p$ of radius $r$ with respect to the Riemannian metric $h$. To distinguish
exponential maps stemming from various metrics we will use a superscript, as in $\exp_p^g$.

To conclude this introductory section, we recall what is known about the exponential map of
$C^{1,1}$ pseudo-Riemannian metric in general. In J.\ H.\ C.\ Whitehead's classical paper 
\cite{W32}, a path 
is a solution of a system of ODEs of the form
$$
\frac{d^2c^k}{dt^2}+\Gamma_{ij}^{k}(c(t))\frac{dc^{i}}{dt}\frac{dc^{j}}{dt}=0,
$$
where the $\Gamma_{ij}^{k}$ are merely supposed to be Lipschitz and symmetric in $i$, $j$
(but are not necessarily the Christoffel symbols of some metric).
In \cite[Sec.\ 3]{W32} it is proved that under these mild assumptions every point $p\in M$ has a neighborhood 
$S$ which is a {\em simple region}, in the sense that any two points in $S$ can be connected by at most one path. 
In particular, this result applies to the geodesics of a $C^{1,1}$ pseudo-Riemannian metric $g$ of arbitrary signature.
It follows that 
$\exp^g_p: (\exp_p^g)^{-1}(S)\to S$ is continuous and bijective, hence a homeomorphism
by invariance of domain, and one has:
\begin{Theorem}\label{mainsemi}
Let $M$ be a smooth manifold with a $C^{1,1}$ pseudo-Riemannian metric $g$ and let $p\in M$. Then there exist open neighborhoods
$U$ of $0\in T_{p}M$ and $V$ of $p$ in M such that 
\begin{equation*}
 \exp^g_{p}:U\rightarrow\ V
\end{equation*}
is a homeomorphism.
\end{Theorem}
Our aim thus is to strengthen this result by additionally establishing the bi-Lipschitz property of 
$\exp_p^{g}$. We note, however, that our proof is self-contained and will not pre-suppose Th.\ \ref{mainsemi}.
Rather, it implicitly provides an alternative proof for this result.

\section{The main result}\label{mainsec}
The aim of this section is to prove the following result:
\begin{Theorem}\label{mainpseudo}
Let $M$ be a smooth manifold with a $C^{1,1}$ pseudo-Riemannian metric $g$ and let $p\in M$. Then there exist open neighborhoods
$U$ of $0\in T_{p}M$ and $V$ of $p$ in M such that 
\begin{equation*}
 \exp^g_{p}:U\rightarrow\ V
\end{equation*}
is a bi-Lipschitz homeomorphism.
\end{Theorem}

Since the result is local, we may assume $M=\R^{n}$  
and $p=0$\footnote{Nevertheless we will
write $p$ below to distinguish considerations in $T_pM$ from those in $M$.}. By $g_E$ or $\langle \,,\, \rangle_E$ we denote
the standard Euclidean metric on $\R^n$, and we write $\|\ \|_E$ for the corresponding standard
Euclidean norm, as well as for mapping norms induced by the Euclidean norm.

As was already indicated in the introduction, our strategy of proof will be to approximate $g$ 
by a net $g_\eps$ of smooth pseudo-Riemannian metrics and then use comparison results to control 
the relevant geometrical quantities derived from
the $g_\eps$ uniformly in $\eps$ so as to preserve the bi-Lipschitz property as $\eps\to 0$.

Thus take $\rho\in \mathscr{D}(\R^{n})$ with unit integral and define the standard mollifier
$\rho_{\eps}:=\eps^{-n}\rho\left (\frac{x}{\eps}\right)$ ($\eps>0$). 
We set $g_{\eps}:=g*\rho_{\eps}$ (componentwise convolution). 
\begin{remark}\label{convergence} For later reference, we note the following properties of the approximating
net $g_\eps$.
\begin{enumerate}
 \item[(i)] 
$g_\eps \to g$ in $C^1(M)$ and the second derivatives of $g_\eps$
 are bounded, uniformly in $\eps$, on compact sets.
\item[(ii)] On any compact subset of $M$, for $\eps$ sufficiently small the $g_\eps$ are a 
family of pseudo-Riemannian metrics of the same signature as $g$
whose Riemannian curvature tensors $R_\eps$ are bounded uniformly in $\eps$.
\end{enumerate}
\end{remark}
In order to proceed we need to determine a neighborhood of $0$ in $T_pM$ that is a 
common domain for all $\exp_{p}^{g_{\eps}}$ for $\eps$ sufficiently small.
Here, and in several places later on, we will make use of 
the following consequence of a standard result on the comparison
of solutions to ODE \cite[10.5.6, 10.5.6.1]{Die}:
\begin{Lemma}\label{dieudonne} 
Let $F,G\in C(H,X)$ where $H$ is a convex open subset of a Banach space $X$. Suppose:
\begin{equation*}
\sup_{x\in H}\|F(x)-G(x)\|\leq \alpha, 
\end{equation*}
$G$ is Lipschitz continuous on $H$ with Lipschitz constant $\le k$, and $F$ is locally Lipschitz on $H$.
For $\mu>0$, define
\begin{equation*}
\varphi(\xi):=\mu e^{\xi k}+\alpha(e^{\xi k}-1)\frac{1}{k},\ \xi\geq 0. 
\end{equation*}
Let $x_{0}\in H$, $t_0\in \R$ and let $u$ be a solution of $x'=G(x)$ with $u(t_{0})=x_{0}$ defined on $J:=(t_{0}-b,t_{0}+b)$ such that $\forall t\in J,\ \overline{B(u(t),\varphi(|t-t_{0}|))}\subseteq H$.
Then for every $\tilde{x}\in H$ with $\|\tilde{x}-x_{0}\|\leq \mu$ there exists a unique solution $v$ of $x'=F(x)$ with $v(t_{0})=\tilde{x}$ on $J$ with values in $H$. 
Moreover, $\|u(t)-v(t)\|\leq \varphi(|t-t_{0}|)$ for $t\in J$. 
\end{Lemma}
We rewrite the geodesic equation for the metric $g$ as a first order system:
\begin{equation}\label{geodesic_equation}
\begin{aligned}
\frac{dc^{k}}{dt}&=y^{k}(t) \\
\frac{dy^{k}}{dt}&=-\Gamma_{g,ij}^{k}(c(t))y^{i}(t)y^{j}(t) 
\end{aligned}
\end{equation}
and analogously for the metrics $g_{\eps}$. Hence, $\exp_{p}^{g}(v)=c(1)$ where $c(0)=p,\ y(0)=v$. Let $t_{0}=0$ and $x_{0}=(p,0)$.
We fix $b>1$ and set $J=(-b,b)$. Now take $u$ to be the constant solution of \eqref{geodesic_equation} with initial condition $x_{0}=(p,0)$, let $\delta>0$ and set $H:=B(x_{0},2\delta)\subseteq \R^{2n}$.
The Christoffel symbols $\Gamma_{g}$ are Lipschitz functions on $H$, and by Remark \ref{convergence} (i)
it follows that there is a common Lipschitz constant
$k$ for $\Gamma_g$ and the $\Gamma_{g_{\eps}}$ on $H$. Choose $\alpha>0,\ \mu>0$ such that
\begin{equation*}
\varphi(b)=\mu e^{bk}+\frac{\alpha}{k}(e^{bk}-1)<\delta. 
\end{equation*}
and choose $\eps_{0}>0$ such that $\forall \eps<\eps_{0}$ we have
$\sup_{H}\| \Gamma_{g}-\Gamma_{g_{\eps}} \| \leq \alpha$. 
Then $\overline{B(u(t),\varphi(|t|))}\subseteq H,\ \forall t\in J$. 
By Lemma \ref{dieudonne}, for all $\tilde{x}=(p,w)\in H$ with $\|\tilde{x}-x_{0}\|=\|w\|\leq \mu$, there exists a unique
solution $u_{\eps}$ on $(-b,b)$ of
\begin{align*}
 \frac{dc^{k}}{dt}&=y^{k}(t) \\
\frac{dy^{k}}{dt}&=-\Gamma_{g_{\eps},ij}^{k}(c(t))y^{i}(t)y^{j}(t), 
\end{align*}
with values in $H$ and $u_{\eps}(0)=\tilde{x}=(p,w)$, as well as a unique solution to \eqref{geodesic_equation}
with these initial conditions. 
Therefore a common domain of $\exp_p^g$ and all $\exp_{p}^{g_{\eps}}$ ($\eps<\eps_0$) is given by
$\{w\in \R^{n}|\ \|w\|_E<\mu \}=:B_E(0,\mu)$.

\begin{remark}\label{estimates} From Remark \ref{convergence} we obtain that for some $\eps_0>0$ we have:
\begin{enumerate}
 \item[(i)] 
There exists a constant $K_1>0$ such that, for $\eps<\eps_0$,
$\|R_\eps\|_{E}\leq  K_1$ uniformly on $B_E(0,\mu)$.
\item[(ii)] For some $K_2>0$ and $\eps<\eps_0$, 
\begin{equation*}
\|\Gamma_{g_{\eps}}\|_E \leq K_2,
\end{equation*}
uniformly on $B_E(0,\mu)$.
\end{enumerate}
\end{remark}

\begin{Lemma}\label{3.2}
Let $r_1 < \min\left(\frac{1}{2 K_2},\frac{1}{2}\mu\right)$. Then for all $\eps<\eps_0$, 
$$
\exp_p^{g_\eps}(\overline{B_E(0,r_1)}) \subseteq B_E(p,\mu).
$$
\end{Lemma}
\begin{proof}
Let $\gamma:[0,r_1]\rightarrow M$ be a $g_\eps$-geodesic with $\gamma(0)=p$ and $\|\gamma'(0)\|_E = 1$
and set $s_0:=\sup\{s\in [0,r_1] \mid \gamma|_{[0,s]}\subseteq B_E(p,\mu)\}$. Then $s_0>0$
and 
for $s\in [0,s_0)$ we have
\begin{align*}
\left |\frac{d}{ds}\langle \gamma'(s),\gamma'(s)\rangle_E \right |&
=2\left |\langle \Gamma_{g_\eps}(\gamma'(s),\gamma'(s)),\gamma'(s)\rangle_E\right| \le 2K_2\|\gamma'(s)\|^{3}_E,
\end{align*}
and therefore 
$
\left |\frac{d}{ds}\|\gamma'(s)\|^{-1}_E\right |\leq K_2.
$
From this, setting $f(s):=\|\gamma'(s)\|_E$, for $s\in [0,s_0)$ we obtain
\begin{equation*}
\frac{f(0)}{f(s)}=\int_{0}^{s}f(0)\frac{d}{d\tau}\left (\frac{1}{f(\tau)}\right )d\tau+1 \in [1/2,3/2],
\end{equation*}
so
\begin{equation}\label{10}
\frac{1}{2}\|\gamma'(0)\|_E\leq  \|\gamma'(s)\|_E\leq 2\|\gamma'(0)\|_E \qquad (s\in [0,s_0)).
\end{equation}
Therefore, 
$$
L_E(\gamma|_{[0,s_0]}) = \int_{0}^{s_0}\|\gamma'(s)\|_{E}\,ds\leq 2r_1 \|\gamma'(0)\|_{E} < \mu,
$$
implying that $s_0=r_1$.
\end{proof}

We next want to determine a ball around $0\in T_pM$ on which each 
$\exp_{p}^{g_{\eps}}$ is a local
diffeomorphism. 
To achieve this, we first need to derive 
estimates on Jacobi fields along geodesics, based on \cite[Sec.\ 4]{Chen}. 

\begin{Lemma}\label{3.4} Set $C_1:=2K_2$, $C_2:=4K_1$ and let 
$$
r_2 < 
\min\left(r_1,\frac{1}{C_1}\log\left(\frac{C_1+C_2}{C_1/2 + C_2}\right),(2+C_1)^{-1}\right).
$$
Then for $\eps<\eps_0$, any $g_\eps$-geodesic $\gamma:[0,r_{2}]\rightarrow M$ with $\gamma(0)=p$ and $\|\gamma'(0)\|_{E}=1$
lies entirely in $B_{E}(p,\mu)$. Moreover, if $J$ is a $g_\eps$-Jacobi field along $\gamma$
with $J(0)=0$ and $\|\nabla_{g_\eps,\gamma'}J(0)\|_E=1$ then 
$\|J(s)\|_E\le 1$ and $\frac{1}{2}\leq \|\nabla_{g_\eps,\gamma'}J\|_E\leq 2$ for all $s\in [0,r_2]$.
\end{Lemma}
\begin{proof} 
By Lemma \ref{3.2}, $\gamma$ lies in $B_{E}(p,\mu)$. Also, \eqref{10} implies
\begin{equation}\label{gammaest}
\max_{s\in[0,r_{2}]}\|\gamma'(s)\|_E \le 2.
\end{equation}
Suppose that $s_0:=\sup\{s\in [0,r_2] \mid \|J(t)\|_E\leq 1 \ \forall t\in [0,s]\} < r_2$.
By assumption, $J$ satisfies
\begin{align*}
\nabla_{g_{\eps},\gamma'}\nabla_{g_\eps,\gamma'} J(s)&=-R_\eps(J(s),\gamma'(s))\gamma'(s)\\
J(0)&=0,\ \ \|J'(0)\|_E=1. 
\end{align*}
Thus by Remark \ref{estimates} and \eqref{gammaest}, on $[0,s_0]$ we obtain
\begin{align*}
\left |\frac{d}{ds}\langle \nabla_{g_\eps,\gamma'}J, \nabla_{g_\eps,\gamma'}J\rangle _E\right |&=
2|\langle \nabla_{g_{\eps},\gamma'}\nabla_{g_\eps,\gamma'}J, \nabla_{g_\eps,\gamma'}J\rangle _E
-\langle \Gamma_{g_\eps}(\nabla_{g_\eps,\gamma'}J,\gamma'),\nabla_{g_\eps,\gamma'}J\rangle_E|\\
&\leq 8K_1\|\nabla_{g_\eps,\gamma'}J\|_E +
4K_2\|\nabla_{g_\eps,\gamma'}J\|_E^{2},
\end{align*}
so that 
\begin{equation}\label{12}
\left |\frac{d}{ds}\|\nabla_{g_\eps,\gamma'}J\|_E\right |\leq
 4K_1 + 2K_2\|\nabla_{g_\eps,\gamma'}J\|_E
=C_1\|\nabla_{g_\eps,\gamma'}J\|_E+C_2.
\end{equation}
Taking into account that $\|\nabla_{g_\eps,\gamma'}J(0)\|_E = 1$ by assumption, integration of
\eqref{12} leads to
\begin{equation*}
-\frac{C_{2}}{C_{1}}+\left (1+\frac{C_{2}}{C_{1}}\right )e^{-C_{1}s}\leq \|\nabla_{g_\eps,\gamma'}J(s)\|_E\leq -\frac{C_{2}}{C_{1}}+\left (1+\frac{C_{2}}{C_{1}}\right )e^{C_{1}s}.
\end{equation*}
Due to our choice of $r_2$, this entails 
\begin{equation}\label{below}
\frac{1}{2}\leq \|\nabla_{g_\eps,\gamma'}J\|_E\leq 2
\end{equation} 
on $[0,s_0]$.
From this, we get
\begin{equation*}
\left|\frac{d}{ds}\|J(s)\|_E \right|
=\frac{1}{\|J(s)\|_E} \left|\langle \nabla_{g_{\eps},\gamma'} J(s),J(s) \rangle_E - 
\langle \Gamma_{g_\eps}(J(s),\gamma'(s)),J(s)\rangle_E \right|
\leq 2+2K_2.
\end{equation*}
Therefore,
\begin{equation}\label{Jest}
\|J(s)\|_E\leq (2+2K_2)s < s/r_2 < 1
\end{equation}
for $s\in [0,s_0]$. For $s=s_0$, this gives a contradiction to the definition of $s_0$.
\end{proof}

\begin{Lemma} There exists some $0<r_3<r_2$ such that, for all $\eps<\eps_0$,
$\exp_p^{g_\eps}$ is a local diffeomorphism on $B_E(0,r_3)$.
\end{Lemma}
\begin{proof}
For any Jacobi field $J$ as in Lemma \ref{3.4} we have:
\begin{align*}
\frac{d}{ds}\langle \nabla_{g_\eps,\gamma'}J,J\rangle_E&=
\langle \nabla_{g_\eps,\gamma'}\nabla_{g_\eps,\gamma'}J,J\rangle_E
 - \langle \Gamma_{g_\eps}(\nabla_{g_\eps,\gamma'}J,\gamma'),J\rangle_E\\
&+\langle \nabla_{g_\eps,\gamma'}J,\nabla_{g_\eps,\gamma'}J\rangle_E
- \langle \nabla_{g_\eps,\gamma'}J,\Gamma_{g_\eps}(J,\gamma')\rangle_E
\end{align*}
Of these four terms, the third one is bounded from below by $1/4$ due to \eqref{below}.
For the others, employing Lemma \ref{3.4} (see \eqref{gammaest}, \eqref{below}, \eqref{Jest}), we obtain for
$s\in [0,r_2]$:
\begin{align*}
&|\langle \nabla_{g_\eps,\gamma'}\nabla_{g_\eps,\gamma'}J,J\rangle_E(s)|=|\langle R_\eps(J,\gamma')\gamma',J\rangle_E(s)|\leq K_1\|\gamma'(s)\|_E^{2}\|J(s)\|_E^2
\leq 4\frac{K_1}{r_2^2}  s^2 \\
&|\langle \Gamma_{g_\eps}(\nabla_{g_\eps,\gamma'}J,\gamma'),J\rangle_E(s)| \leq K_2\|\gamma'(s)\|_E\|\nabla_{g_\eps,\gamma'}J(s)\|_E\|J(s)\|_E\leq 4\frac{K_2}{r_2}s\\
&|\langle \nabla_{g_\eps,\gamma'}J,\Gamma_{g_\eps}(J,\gamma')\rangle_E(s)|\leq K_2\|\gamma'(s)\|_E\|J(s)\|_E\|\nabla_{g_\eps,\gamma'}J(s)\|_E\leq 4\frac{K_2}{r_2} s
\end{align*}
From this we obtain an $r_3 = r_3(r_2,K_1,K_2)<r_2$ such that on $[0,r_3]$, 
$\frac{d}{ds}\langle \nabla_{g_\eps,\gamma'}J,J\rangle_E$ is bounded from below by a positive constant.
By the same estimates and \eqref{below} again, it is also bounded from above. Hence for some $c_1>0$,
any $\eps<\eps_0$ and $s\in [0,r_3]$ we obtain:
\begin{equation*}
e^{-c_{1}}\leq \frac{d}{ds}\langle \nabla_{g_\eps,\gamma'}J,J\rangle_E(s)\leq e^{c_{1}}, 
\end{equation*}
and therefore
\begin{equation*}
e^{-c_{1}}s\leq \langle \nabla_{g_\eps,\gamma'}J,J\rangle_E(s)\leq e^{c_{1}}s.
\end{equation*}
Combined with \eqref{below} and \eqref{Jest}, this entails:
\begin{equation*}
\frac{1}{r_2}s \ge
\|J(s)\|_E \geq \frac{\langle \nabla_{g_\eps,\gamma'}J,J\rangle_E(s)}{\|\nabla_{g_\eps,\gamma'}J(s)\|_E}\geq \frac{e^{-c_1}}{2}s.
\end{equation*}
Altogether, we find $c_2>0$ such that for all $\eps<\eps_0$ and $s\in[0,r_3]$:
\begin{equation*}
 e^{-c_2}s\leq \|J(s)\|_E \leq e^{c_2}s. 
\end{equation*}
In terms of the exponential map, any Jacobi field as in Lemma \ref{3.4} is of the form
$J(s)=T_{s\gamma'(0)}\exp_{p}^{g_{\eps}}(s\cdot w)$, with $w\in T_pM$, $\|w\|_E=1$.
Thus
\begin{align*}
 e^{-c_2}&\leq \|T_{s\gamma'(0)}\exp_{p}^{g_{\eps}}(w)\|_E\leq e^{c_2} \qquad (s\in [0,r_3]).
\end{align*}
Since $\|\gamma'(0)\|_E=1$ we conclude that $\forall \eps<\eps_0$, $\forall v\in B_E(0,r_3)$,
$\forall w\in T_pM$:
\begin{equation}\label{LRauch}
e^{-c_2}\|w\|_E\leq \|T_{v}\exp_{p}^{g_{\eps}}(w)\|_E\leq 
e^{c_2}\|w\|_E.
\end{equation}
In particular, $\exp_p^{g_{\eps}}$ is a local diffeomorphism on $B_E(0,r_3)$.
\end{proof}

We note that \eqref{LRauch} can equivalently be formulated as
\begin{equation}\label{final}
e^{-2c_2}g_E\leq (\exp_{p}^{g_{\eps}})^{\ast}g_E\leq e^{2c_2}g_E 
\end{equation}
for $\eps<\eps_0$ on $B_E(0,r_3)$. 

\begin{Lemma} \label{inclusions}
For $r_{4}<e^{-c_2}r_3$, $r_{5}< e^{-c_2}r_{4}$ and
$\tilde r := e^{c_2}r_{4}$ we have, $\forall \eps < \eps_0$:
\begin{equation*}
 \exp_{p}^{g_{\eps}}(\overline{B_E(0,r_{5})})\subseteq B_E(p,r_{4})\subseteq 
 \exp_{p}^{g_{\eps}}(\overline{B_E(0,\tilde r)})
 \subseteq\exp_{p}^{g_{\eps}}(B_E(0,r_3)).
\end{equation*}
\end{Lemma}
\begin{proof}
For $q\in B_{E}(p,r_4)$, let $\alpha:[0,a]\rightarrow M$ be a piecewise smooth curve from
$p$ to $q$ in $B_{E}(p,r_4)$ of Euclidean length less than $r_4$. Since $\exp_{p}^{g_{\eps}}$ is a local diffeomorphism 
on $B_E(0,r_3)$, for $b>0$ sufficiently small there exists a unique $\exp_{p}^{g_{\eps}}$-lift $\hat\alpha:[0,b]\to B_E(0,r_3)$ of $\alpha|_{[0,b]}$ starting at $0$. We claim that $a':=\sup\{b<a|\ \hat{\alpha} \text{ exists on } [0,b]\} = a$. Indeed, suppose that
$a'<a$. Then  
\begin{align*}
L_{(\exp_{p}^{g_{\eps}})^{\ast}g_E}(\hat{\alpha}|_{[0,a')})=
L_{g_E}(\alpha|_{[0,a')})=\int_{0}^{a'}\|\alpha'(t)\|_{E}dt\leq r_4. 
\end{align*}
Hence by \eqref{final} we obtain $L_{g_E}(\hat{\alpha}|_{[0,a')})\le \tilde r$. Now let $a_{n}\nearrow a'$. Then 
$\hat{\alpha}(a_{n})\in \overline{B_E(0,\tilde r)}$, 
so some  subsequence 
$(\hat\alpha(a_{n_k}))$ converges to a point $v$ in $\overline{B_E(0,\tilde r)}$. Since $\exp_p^{g_\eps}$ is a diffeomorphism
on a neighborhood of $v$ and $\exp_p^{g_\eps}(v) = \lim \alpha(a_{n_k}) = \alpha(a')$, this shows
that $\hat\alpha$ can be extended past $a'$, a contradiction. Thus $q=\exp_{p}^{g_{\eps}}(\hat\alpha(a))\in \exp_{p}^{g_{\eps}}(\overline{B_E(0,\tilde r)})$.

For the first inclusion, take $v\in T_{p}M$ with $\|v\|_E\le r_5$ and set $q:=\exp_{p}^{g_{\eps}}(v)$.
Then for the radial geodesic $\gamma:[0,1]\rightarrow M$, $t\mapsto \exp_p^{g_\eps}(tv)$ from $p$ to $q$, 
by \eqref{LRauch} 
we obtain for $s$ small:
\begin{align*}
L_{E}(\gamma|_{[0,s]})&=\int_{0}^{s}\| T_{tv}\exp^{g_\eps}_p(v)\|_{E}\,dt\le 
e^{c_2} \|v\|_E < r_4
\end{align*}
From this we conclude that $\sup\{s\in [0,1]\mid\gamma|_{[0,s]}\subseteq B_{E}(p,r_4)\} = 1$, so
$q\in B_{E}(p,r_4)$. 
\end{proof}

Note that $\exp_{p}^{g_{\eps}}: \overline{B_E(0,\tilde{r})}\rightarrow \exp_{p}^{g_{\eps}}(\overline{B_E(0,\tilde{r})})$
is a surjective local homeomorphism between compact Hausdorff spaces, hence is a covering map. Using this, we obtain: 
\begin{Lemma}\label{injlemma}
For any $\eps<\eps_{0}$, $\exp_p^{g_{\eps}}$ is injective (hence a diffeomorphism) on $B_E(0,r_5)$.
\end{Lemma}
\begin{proof}
Suppose to the contrary that there exist $v_{0},v_{1}\in B_E(0,r_5),\ v_{0}\neq v_{1},$ and $\eps<\eps_{0}$ such that 
$\exp_{p}^{g_{\eps}}(v_{0})=q=\exp_{p}^{g_{\eps}}(v_{1})$. Hence, $\gamma_{i}(t):=\exp_{p}^{g_{\eps}}(tv_{i}),\ i=0,1,$ are two distinct geodesics starting at
$p$ which intersect at the point $q$. Then $\gamma_{s}(t)=s\gamma_{1}(t)+(1-s)\gamma_{0}(t)$
is a fixed endpoint homotopy connecting
$\gamma_{0}$ and $\gamma_{1}$ in the ball $B_{E}(p,r_4)$. 
Since $\exp_{p}^{g_{\eps}}$ is a covering map, and using Lemma \ref{inclusions},
we can lift this homotopy to $\overline{B_E(0,\tilde{r})}$. But the lifts of $\gamma_{0}$ and $\gamma_{1}$ are $t\mapsto tv_{i},\ i=0,1$, which obviously are not 
fixed endpoint homotopic in $\overline{B_E(0,\tilde{r})}$, a contradiction.
\end{proof}

From \eqref{LRauch} we obtain a uniform Lipschitz constant for all $\exp_p^{g_\eps}$ with $\eps<\eps_0$:
$\exists c_3>0$ such that $\forall u,v\in B_E(0,r_5)$
\begin{equation*}
\|\exp_{p}^{g_{\eps}}(u)-\exp_{p}^{g_{\eps}}(v)\|_{E}\leq c_3\|u-v\|_E. 
\end{equation*}

For the corresponding estimate from below we use the following result that provides a mean value
estimate for $C^1$-functions on not necessarily convex domains (cf.\ \cite[3.2.47]{GKOS}).
\begin{Lemma}\label{Kriegl}
Let $\Omega\subseteq\R^n$, $\Omega'\subseteq \R^m$ be open, 
$f\in C^1(\Omega,\Omega')$ and suppose that $K\Subset\Omega$. 
Then there exists $C>0$, such that $\|f(x)-f(y)\|\leq C\|x-y\|,\ \forall x,y\in K$. $C$ can be chosen
as $C_1\cdot \sup_{x\in L}(\|f(x)\|+\|Df(x)\|)$ for any fixed compact neighborhood $L$ of $K$ in $\Omega$,
where $C_1$ depends only on $L$.
\end{Lemma}

Using Lemma \ref{inclusions} we now pick $r_7<r_6 =: e^{-c_2}\hat r < \hat r <r_5$ such that for $\eps<\eps_0$ we have 
\begin{equation*}
 \exp_{p}^{g_{\eps}}(\overline{B_E(0,r_7}))\subseteq B_E(p,r_6)\subseteq 
 \exp_{p}^{g_{\eps}}(\overline{B_E(0,\hat r)})\Subset \exp_{p}^{g_{\eps}}(B_E(0,r_5)).
\end{equation*}
Again by \eqref{LRauch}, we have $\forall \eps<\eps_0$,
\begin{equation*}
 e^{-c_2}\|\xi\|_E \leq \|T_{q}(\exp_{p}^{g_{\eps}})^{-1}(\xi)\|_E \leq e^{c_2}\|\xi\|_E,
\end{equation*}
$\forall q\in \overline{B_E(p,r_6)},\forall \xi\in T_qM$. Thus by Lemma \ref{Kriegl} there exists some $c_4>0$ such that
\begin{equation*}
\|(\exp_p^{g_\eps})^{-1}(q_1) - (\exp_p^{g_\eps})^{-1}(q_2)\|_E \le c_4^{-1} \|q_1-q_2\|_{E}
\end{equation*}
$\forall\eps<\eps_0$, $\forall q_1,q_2\in \exp_p^{g_\eps}(B_E(0,r_7))$.

Summing up, for all $\eps<\eps_0$ and all $u,v\in B_E(0,r_7)$ we have
\begin{equation*}
 c_4\|u-v\|_E \leq \|\exp_{p}^{g_{\eps}}(u)-\exp_{p}^{g_{\eps}}(v)\|_E \leq c_3\|u-v\|_E. 
\end{equation*}
Finally, let $\eps \rightarrow 0$. Then for all $u,v\in B_E(0,r_7)$ we get
\begin{equation*}
  c_4\|u-v\|_E \leq \|\exp_{p}^{g}(u)-\exp_{p}^{g}(v)\|_E \leq c_3\|u-v\|_E. 
\end{equation*}
Thus, $\exp_{p}^{g}$ is a bi-Lipschitz homeomorphism on $U:=B_E(0,r_7)\subseteq T_pM$. In particular,
$V=\exp^g_p(U)$ is open in $M$ (invariance of domain). This concludes
the proof of Theorem \ref{mainpseudo}. 
\section{The Riemannian case}\label{Riemsec}
In the special case where $g$ is a $C^{1,1}$ Riemannian metric, in this section we point
out some alternatives to the reasoning given in the previous section.
We start out from the following version of the Rauch comparison theorem, cf., e.g., \cite[Cor.\ 4.6.1]{Jost}.
\begin{Theorem}\label{Rauch}
Let $(M,h)$ be a smooth Riemannian manifold and suppose that $\exp_{p}^{h}$ is defined on a ball $B_{h_{p}}(0,R)$, for some $R>0$, and that there exist
$\rho \leq 0,\ \kappa > 0$ such that
the sectional curvature K of M satisfies $\rho \leq K \leq \kappa$ on some open set 
which contains
$\exp_{p}^{h}(B_{h_{p}}(0,R))$.
Then for all $v\in T_{p}M$ with $\|v\|_{h_{p}}=1$, 
all $w\in T_pM$, and all $0< t < \min(R,\frac{\pi}{\sqrt{\kappa}})$,
\begin{equation*}
\frac{\mathrm{sn}_\kappa(t)}{t}\|w\| \leq \|(T_{tv}\exp_{p}^{h})(w)\| \le \frac{\mathrm{sn}_\rho(t)}{t}\|w\|.
\end{equation*}
\end{Theorem}
Here, for $\alpha\in \R$,
$$
\mathrm{sn}_\alpha(t) := \left\{
\begin{array}{lr}
	 \frac{1}{\sqrt{\alpha}}\sin(\sqrt{\alpha} t) \quad & \text{for}\ \alpha>0\\
	 t \quad & \text{for}\ \alpha=0\\
	 \frac{1}{\sqrt{-\alpha}}\sinh(\sqrt{-\alpha} t) \quad & \text{for}\ \alpha<0
\end{array}
\right.
$$
As an immediate consequence, we obtain that for any $0<r<\min(R,\frac{\pi}{\sqrt{\kappa}})$, there exists
some $c>0$ such that $\forall v\in B_{h_{p}}(0,r)$, $\forall w\in T_pM$
\begin{equation}\label{rauchequ}
e^{-c}\|w\| \leq \|(T_{v}\exp_{p}^{h})(w)\| \leq e^{c}\|w\|.
\end{equation}
\begin{remark}\label{Grassmann} For $(M,g)$ a smooth Riemannian manifold, its sectional curvature $K$ is a 
smooth function on the $2$-Grassmannian bundle $G(2,TM)$. Since the fibers of $G(2,TM)$ are compact,
local bounds on the Riemann curvature tensor $R$ imply local bounds on $K$ on any relatively compact
subset of $M$. However, an analogous argument is not possible in the Lorentzian (or
general pseudo-Riemannian) setting since in that case $K$ is only defined on non-degenerate $2$-planes,
forming an open subbundle of $G(2,TM)$. Indeed, a Lorentzian manifold has bounded sectional curvature
$K$ only in the trivial case where $K$ is constant (\cite{K79}, cf.\ also \cite{Harris}). 
\end{remark}
Now let $g$ be a $C^{1,1}$ Riemannian metric $M$, and let $g_\eps$ be approximating smooth
metrics as in Section \ref{mainsec}. Then we may fix some $r'>0$ and some $\eps_0>0$
such that $\exp_p^g$ and $\exp_{p}^{g_{\eps}}$ ($\eps<\eps_0$)
are defined on $B_{g_{p}}(0,r')$ and such that
(by locally uniform convergence of $\exp_p^{g_\eps}$ to $\exp_p^g$)
there exists an open, relatively compact subset $W\subseteq M$ with
$\bigcup_{\eps<\eps_{0}}\exp_{p}^{g_{\eps}}(B_{g_{p}}(0,r'))\subseteq W$. 
On $W$, by Remarks \ref{convergence} (ii) and \ref{Grassmann} we obtain uniform bounds on the sectional curvatures 
$K_{\eps}$ of $g_{\eps}$, i.e.,
\begin{equation*}
\exists \rho\le 0,\kappa > 0: \forall \eps<\eps_0 \ \ \rho\leq K_{\eps}\leq \kappa. 
\end{equation*}
Thus by \eqref{rauchequ}, for any $r<\min(r',\frac{\pi}{\sqrt{\kappa}})$, there exists some $c>0$ 
depending only on $\rho$ and $\kappa$ such that for all $\eps<\eps_0$
\begin{equation}\label{2}
 e^{-c}\|w\|_{g_\eps} \leq \|(T_{v}\exp_{p}^{g_{\eps}})(w)\|_{g_\eps} \leq e^{c}\|w\|_{g_\eps},
\end{equation}
$\forall v\in B_{g_{p}}(0,r),\forall w\in T_{p}M$.
In particular, by the inverse function theorem every $\exp_{p}^{g_{\eps}}$ is a local diffeomorphism on $B_{g_{p}}(0,r)$. 
Thus we may rewrite \eqref{2} equivalently as
\begin{equation*}
e^{-2c}g_{\eps,p}\leq (\exp_{p}^{g_{\eps}})^{\ast}g_{\eps}\leq e^{2c}g_{\eps,p},
\end{equation*}
on $B_{g_{p}}(0,r)$. Since $g_\eps\to g$ locally uniformly, by increasing $c$ 
we obtain \eqref{final} on a suitable Euclidean ball and can proceed as in Section \ref{mainsec}.

Finally, we note that  
to obtain a common domain (and injectivity) of the
approximating exponential maps $\exp_{p}^{g_{\eps}}$ one may alternatively employ the following
result of Cheeger, Gromov and Taylor (\cite{CGT82}, the formulation below is taken from 
\cite{Chen}), which provides a lower bound on the injectivity radii
$\Inj_{g_{\eps}}(M,p)$. 
\begin{Theorem}\label{cgt}
Let $M$ be a $C^\infty$ $n$-manifold with a smooth Riemannian metric $g$. Suppose that 
$\overline{B_g(p,1)}\Subset M$ for some point $p$ in $M$. 
Then for any $K,v>0$ there exists some $i=i(K,v,n)>0$ such that if
\begin{equation*}
 \|R_{g}\|_{L^{\infty}(B(p,1))}\leq K,\  \Vol _{g}(B(p,1))\geq v,
\end{equation*}
then the injectivity radius $\Inj_{g}(M,p)$ at $p$ is bounded from below by $i$,
\begin{equation*}
 \Inj_{g}(M,p)\geq i.
\end{equation*}
\end{Theorem}

Since the distance function $d_g$ of the $C^{1,1}$-metric $g$ induces the manifold topology,
$B_g(p,2r)$ is an open, relatively compact subset of $M$ for $r>0$ sufficiently small.
Thus for $\eps$ small, $B_{g_\eps}(p,r)\subseteq B_g(p,2r)$ is relatively compact and
$$
\Vol_{g_\eps}(B_{g_\eps}(p,r)) \ge \Vol_{g_\eps}(B_{g}(p,r/2)) \ge \frac{1}{2}\Vol_{g}(B_{g}(p,r/2))>0.
$$
By Theorem \ref{cgt}, there exists some $r_{0}$ such that
\begin{equation*}
 \Inj(g_{\eps},p)\geq r_{0},\ \forall \eps \leq \eps_{0},
\end{equation*}
so $\exp_{p}^{g_{\eps}}$ is a diffeomorphism on $B_{g_{\eps}}(p,r_{0})\ \forall \eps \leq \eps_{0}$.
Since $B_{g}(p,\frac{r_{0}}{2})\subseteq B_{g_{\eps}}(p,r_{0})$ for $\eps$ small, it follows that $\exp_{p}^{g_{\eps}}$ is a diffeomorphism on $B_{g}(p,\frac{r_{0}}{2})$.
From here, using Theorem \ref{Rauch}, we may proceed as in the argument following Lemma \ref{injlemma} to conclude 
that $\exp_{p}^{g}$ is a bi-Lipschitz homeomorphism on some neighborhood of $0\in T_pM$.

\section{Totally normal neighborhoods}
For a smooth pseudo-Riemannian metric $g$ on a manifold $M$, a neighborhood $U$ of $p\in M$
is called a normal neighborhood of $p$ if $\exp^g_p$ is a diffeomorphism from a starshaped
open neighborhood $\tilde U$ of $0\in T_pM$ onto $U$. $U$ is called {\em totally normal}
if it is a normal neighborhood of each of its points. This terminology is in line with
\cite{DoCarmo2} while, e.g., in \cite{ON83} such sets are called geodesically convex.

Analogously, if $g$ is a $C^{1,1}$-pseudo-Riemannian metric on a smooth manifold $M$
we call a neighborhood of a point $p\in M$ normal if there exists a starshaped
open neighborhood $\tilde U$ of $0\in T_pM$ such that $\exp^g_p$ is a bi-Lipschitz
homeomorphism from $\tilde U$ onto $U$. $U$ is called totally normal 
if it is a normal neighborhood of each of its points.

In what follows we adapt the standard proof for the existence of totally normal neighborhoods,
cf., e.g., \cite[Prop.\ 5.7]{ON83} (tracing back to \cite[Sec.\ 4]{W32}) to the $C^{1,1}$-situation.

\begin{Theorem} \label{totally}
Let $M$ be a smooth manifold with a $C^{1,1}$ pseudo-Riemannian metric $g$. Then each point $p\in M$ 
possesses a basis of totally normal neighborhoods.   
\end{Theorem}
\begin{proof}
The main point to note is that the explicit bounds derived in Section \ref{mainsec}  
on the radius of the ball in 
$T_pM$ where $\exp_p^g$ is a bi-Lipschitz homeomorphism depend only on quantities that can be uniformly 
controlled on compact sets. Therefore, for any $p\in M$ there
exists a neighborhood $V'$ of $p$ and some $r>0$ 
such that, $\forall q\in V'$,
\begin{equation}\label{uniform}
\exp_{q}^{g}: B_{h,q}(0,r)\rightarrow  \exp_{q}^{g}(B_{h,q}(0,r))
\end{equation}
is a bi-Lipschitz homeomorphism. Here, $h$ is any background
Riemannian metric.

Now define $S:=\{v\in TM\ |\ \pi(v)\in V', \|v\|_{h}<r\}$, with $\pi$ the natural projection
of $TM$ onto $M$. Let $E:TM\rightarrow M\times M$, $E(v)=(\pi(v), \exp^g(v))$. Then by \eqref{uniform}
$E: S\to E(S)=:W$ is a continuous bijection, hence a homeomorphism by invariance of domain.
Let $(\psi=(x^{1},...,x^{n}),V)$ be a coordinate system centered at $p$
(in the smooth case $\psi$ is usually taken to be a normal coordinate system, which is not available to us, 
but this is in fact not needed). Define the $(0,2)$-tensor field $B$ on $V$  by
 \begin{equation*}
 B_{ij}(q):=\delta_{ij}-\sum_{k}\Gamma_{ij}^{k}(q)x^{k}(q).  
 \end{equation*}
Since $\psi(p)=0$ we may assume $V$ small enough that $B$ is positive definite on $V$.
In addition, we may suppose that $W\subseteq V\times V$. 
Set $N(q):=\sum_{i=1}^{n}(x^{i}(q))^{2}$, and let $V(\delta):=\{q\in V\ | \ N(q)<\delta \}$.
Then if $\delta$ is so small that $V(\delta)\times V(\delta)\subseteq W$, 
$E$ is a homeomorphism from $U_{\delta}:=E^{-1}(V(\delta)\times V(\delta))$ onto $V(\delta)\times V(\delta)$
and $\exp^g([0,1]\cdot U_\delta)\subseteq \exp^g(S)\subseteq V$.

We will show that $V(\delta)$ is totally normal.
For $q\in V(\delta)$ and $U_{q}:=U_{\delta}\cap T_{q}M$, $\exp_q^g=E|_{U_{q}}: U_{q}\rightarrow V(\delta)$ is a homeomorphism, so it is left to show that $U_{q}$ is starshaped.
Let $v\in U_q$. Then $\sigma: [0,1]\rightarrow M$, 
$\sigma(t)=\exp_{q}^{g}(tv)$ is a geodesic from $q$ to $\sigma(1)=:\tilde{q}\in V(\delta)$ that lies entirely in $V$.

If $\sigma$ is contained in $V(\delta)$ then $tv\in U_{q}$, $\forall t\in [0,1]$: suppose to the contrary
that $\bar{t}:=\sup\{t\in [0,1]|\ [0,t]\cdot v\in U_{q}\}<1$. Then $\bar{t}v\in \partial U_{q}$ and since  
$(\exp_{q}^{g}|_{U_{q}})^{-1}(\sigma([0,1])) \Subset U_{q}$, there exists some $t_{1}<\bar{t}$ 
such that $U_{q}\ni t_{1}v \notin (\exp_{q}^{g}|_{U_{q}})^{-1}(\sigma([0,1]))$,  
a contradiction.
Hence the entire segment $\{tv|\ t\in [0,1]\}$ lies in $U_{q}$, so $U_{q}$ is starshaped.
It remains to show that $\sigma$ cannot leave $V(\delta)$. If it did, there would exist $t_{0}\in [0,1]$ such that $N(\sigma(t_{0}))\geq \delta$. Since $N(q), N(\tilde{q})<\delta$, 
the function $t\mapsto N\circ \sigma$ has a maximum at some point $\tilde{t}\in (0,1)$. 
However,
\begin{equation*}
\frac{d^{2}(N\circ\sigma)}{dt^{2}} (\tilde t)
=2B_{\sigma(\tilde{t})}((\psi\circ\sigma)'(\tilde{t}),(\psi\circ\sigma)'(\tilde{t}))>0, 
	\end{equation*}
a contradiction. 
\end{proof}

{\bf Acknowledgements.} We are grateful to Anton Petrunin, who put us on the right track when we first
posted the question that led to this article on MathOverflow. We also would like to thank 
Piotr Chru\'sciel, James D.\ E. Grant, Stefan Haller, and James A.\ Vickers for helpful discussions. 
The authors acknowledge the support
of FWF projects P23714 and P25326, as well as OeAD project WTZ CZ 15/2013.

\end{document}